\documentclass[reqno,a4paper,11pt]{amsart}

\usepackage[utf8]{inputenc}
\usepackage{amsmath}
\usepackage{amssymb}
\usepackage{amsthm}
\usepackage{mathtools}

\usepackage{times}
\usepackage{amsaddr}
\usepackage{enumitem}
\usepackage{cite}
\usepackage{tikz}
\usepackage{cleveref}
\usepackage[english]{babel}
\usepackage[T1]{fontenc}
\usepackage{fourier}
\usepackage{fullpage}
\usepackage{caption}

\newcommand{\whp}{whp}%with high probability}
\newcommand{\prob}[1]{\mathbb{P}\left[#1\right]} %Probability
\newcommand{\condprob}[2]{\mathbb{P}\left[#1 \;\middle|\; #2\right]}

\def\Erdos{Erd\H{o}s}
\def\Renyi{R\'enyi}

\def\ER{\Erdos-\Renyi}

\newtheorem{thm}{Theorem}[section]
\newtheorem{prop}[thm]{Proposition}
\newtheorem{coro}[thm]{Corollary}
\newtheorem{lem}[thm]{Lemma}

\theoremstyle{remark}
\newtheorem{remark}[thm]{Remark}
\theoremstyle{definition}
\newtheorem{definition}[thm]{Definition}

\newcommand{\proofofW}[1]{\subsection*{Proof of \Cref{#1}}}

\crefname{thm}{theorem}{theorems}
\crefname{prop}{proposition}{propositions}
\crefname{coro}{corollary}{corollaries}
\crefname{lem}{lemma}{lemmas}
\crefname{definition}{definition}{definitions}
\crefname{question}{question}{questions}

\newtheoremstyle{claim}%name
{}%Space above
{}%Space below
{\itshape}%Body font
{}%Indent amount
{\bf}%Theorem head font
{.}%Punctuation after theorem head
{.5em}%Space after theorem head
{}%Theorem head spec(can be left empty, meaning ‘normal’)
\theoremstyle{claim}

\crefname{claim}{claim}{claims}

\def\N{\mathbb{N}} %natural numbers
\newcommand{\smallo}[1]{o\left(#1\right)}
\newcommand{\bigo}[1]{O\left(#1\right)}
\newcommand{\smallomega}[1]{\omega\left(#1\right)}
\newcommand{\Th}[1]{\Theta\left(#1\right)}

\newcommand{\maxdegree}[1]{\Delta\left(#1\right)} %maximum degree
\newcommand{\largestcomponent}[1]{L\left(#1\right)} %largest component
\def\Largestcomponent{L} %largest component
\newcommand{\vertexSet}[1]{V\left(#1\right)} %vertex set of a graph
\newcommand{\edgeSet}[1]{E\left(#1\right)} %edge set of a graph
\newcommand{\numberVertices}[1]{v\left({#1}\right)} %number of vertices in a graph
\newcommand{\numberEdges}[1]{e\left({#1}\right)} %number of edges in a graph
\newcommand{\setbuilder}[2]{\left\{#1 \mid #2\right\}} %set-builder notation

\def\cl{\mathcal{P}}
\newcommand{\pro}[2]{P_{#1,#2}} %graph process (graph after #2 edges have been tested)
\newcommand{\acc}[2]{P_{#1,m=#2}} %graph process (graph after #2 edges have been accepted)
\newcommand{\er}[2]{G_{#1,#2}} %graph process (graph after #2 edges have been considered and all of them were added)

\newcommand{\sol}[1]{\beta\left(#1\right)}
\newcommand{\func}[1]{f\left(#1\right)}
\newcommand{\invFunc}[1]{f^{-1}\left(#1\right)}

\DeclareMathOperator{\excess}{ex}
\newcommand{\ex}[1]{\excess\left(#1\right)}
\newcommand{\weight}[1]{w\left(#1\right)}
\newcommand{\core}[1]{C\left(#1\right)} %core
\newcommand{\forest}[1]{F\left(#1\right)} 
\DeclareMathOperator{\binDistribution}{Bin}
\newcommand{\bin}[2]{\binDistribution\left(#1,#2\right)} 
\DeclareMathOperator{\numberTreeComponents}{\#t}
\newcommand{\nt}[1]{\numberTreeComponents\left(#1\right)} 
\DeclareMathOperator{\considered}{q}
\newcommand{\con}[1]{\considered\left(#1\right)}
\DeclareMathOperator{\forbidden}{forb}
\newcommand{\forb}[1]{\forbidden\left(#1\right)}
\DeclareMathOperator{\addable}{add}
\newcommand{\add}[1]{\addable\left(#1\right)}
\newcommand{\rej}[1]{r\left(#1\right)}

\title{The early evolution of the random graph process in planar graphs and related classes}
\author{Mihyun Kang, Michael Missethan}
\address{Institute of Discrete Mathematics, Graz University of Technology, Steyrergasse 30, 8010 Graz, Austria}
\email{\{kang,missethan\}@math.tugraz.at}
\thanks{Supported by Austrian Science Fund (FWF): I3747 and W1230}
\begin{document}
	
\begin{abstract}
We study the random planar graph process introduced by Gerke, Schlatter, Steger, and Taraz [The random planar graph process, Random Structures Algorithms 32 (2008), no. 2, 236--261; MR2387559]: Begin with an empty graph on $n$ vertices, consider the edges of the complete graph $K_n$ one by one in a random ordering, and at each step add an edge to a current graph only if the graph remains planar. They studied the number of edges added up to step $t$ for \lq large\rq\ $t=\smallomega{n}$. In this paper we extend their results by determining the asymptotic number of edges added up to step $t$ in the early evolution of the process when $t=\bigo{n}$. We also show that this result holds for a much more general class of graphs, including outerplanar graphs, planar graphs, and graphs on surfaces.
\end{abstract}
	
\maketitle
	
\section{Introduction and results}\label{sec:intro}
\subsection{Motivation}\label{subsec:motivation}
\Erdos\ and \Renyi\ \cite{ErdoesRenyi1959,ErdoesRenyi1960} introduced the classical random graph process $\left(\er{n}{t}\right)_{t=0}^{N}$, where one starts with an empty graph on vertex set $[n]:=\left\{1, \ldots, n\right\}$ and adds the $N:=\binom{n}{2}$ many edges of the complete graph $K_n$ one after another in a random order. Since then, many exciting results on $\er{n}{t}$ have been obtained (see e.g., \cite{Bollobas2001,JansonLuczakRucinski2000,FriezeKaronski2016} for an overview), and $\er{n}{t}$ is also known as the \ER\ random graph, because it has the same distribution as the uniform random graph on $[n]$ with exactly $t$ edges.

A variant of the \ER\ random graph process is the $\cl$-constrained random graph process, where an edge is added only when a certain graph property $\cl$ is preserved. More formally, given $n\in\N$ and a graph property $\cl$, i.e., a class of graphs with specific properties, we choose a random ordering $e_1, \ldots, e_{N}$ of the edges of the complete graph $K_n$. Then we let $\pro{n}{0}$ be the empty graph on vertex set $[n]$. For $t\in\left[N\right]$, we set $\pro{n}{t}=\pro{n}{t-1}+e_t$ and say that $e_t$ is \textit{accepted} if $\pro{n}{t-1}+e_t\in \cl$; otherwise, we set $\pro{n}{t}=\pro{n}{t-1}$ and say that $e_t$ is \textit{rejected}. Furthermore, we say that the edge $e_t$ is \textit{queried} at step $t$. We denote by $\numberEdges{\pro{n}{t}}$ the number of edges accepted until step $t$ and by $\rej{t}:=t-\numberEdges{\pro{n}{t}}$ the number of rejected edges.

Prominent examples of a graph property $\cl$ for which the $\cl$-constrained random graph process has been extensively studied include triangle-freeness \cite{ErdoesSuenWinkler1993,Bohman2009,FizPontiverosGriffithsMorris2020}, and more generally $H$-freeness for a fixed graph $H$ \cite{OsthusTaraz2001,BohmanKeevash2010,BollobasRiordan2000}, or having bounded maximum degree \cite{RucinskiWormald1992,RucinskiWormald1997}; these are all \lq local\rq\ properties. More \lq global\rq\ properties have also been considered, such as planarity \cite{GerkeSchlatterStegerTaraz2008}, $k$-colourability \cite{KrivelevichSudakovVilenchik2009}, $k$-matching-freeness \cite{KrivelevichKwanLohSudakov2018}, and the {K}{\H{o}}nig property \cite{KamcevKrivelevichMorrisonSudakov2020}. Most of the obtained results are on properties of the final graph $\pro{n}{N}$ and much less is known about the \lq evolution\rq\ of these processes.

Gerke, Schlatter, Steger, and Taraz \cite{GerkeSchlatterStegerTaraz2008} considered the $\cl$-constrained random graph process for the property $\cl$ of being planar. Among other interesting results, they showed the following.

\begin{thm}[{\cite[Theorem 1.1]{GerkeSchlatterStegerTaraz2008}}]\label{thm:Gerke}
Let $\cl$ be the class of planar graphs and $\left(\pro{n}{t}\right)_{t=0}^{N}$ the $\cl$-constrained random graph process. For every $\varepsilon>0$ there exists $\delta>0$ such that
\begin{align*}
\prob{\numberEdges{\pro{n}{\delta n^2}}\geq \left(1+\varepsilon\right)n}<e^{-n}.
\end{align*}
\end{thm}

\Cref{thm:Gerke} immediately implies the following result on the asymptotic number of edges accepted until a superlinear step. Throughout the paper, we will use standard Landau notation for asymptotics and all the asymptotics are taken as $n\to \infty$. We say that an event holds {\em with high probability} (\whp\ for short) if it holds with probability tending to one as $n\to\infty$.
	
\begin{coro}\label{coro:numberEdges}
Let $\cl$ be the class of planar graphs. Let $\left(\pro{n}{t}\right)_{t=0}^{N}$ be the $\cl$-constrained random graph process and $t=t(n)\in\left[N\right]$ be such that $n\ll t\ll n^2$. Then \whp\ $\numberEdges{\pro{n}{t}}=\left(1+\smallo{1}\right)n$.
\end{coro}

We note that the upper bound in \Cref{coro:numberEdges}, i.e., $\numberEdges{\pro{n}{t}}\leq \left(1+\smallo{1}\right)n$, follows directly from \Cref{thm:Gerke}. Furthermore, it is well known that \whp\ the largest component of $\er{n}{t}$ has $\left(1+\smallo{1}\right)n$ vertices if $t\gg n$ (see e.g., \cite{ErdoesRenyi1960}). Together with the simple fact that the number of vertices in the largest components of $\er{n}{t}$ and $\pro{n}{t}$ coincide (see \Cref{rem:component_structure}) this implies the lower bound on $\numberEdges{\pro{n}{t}}$ in \Cref{coro:numberEdges}.

Gerke, Schlatter, Steger, and Taraz asked the asymptotic behaviour of $\numberEdges{\pro{n}{t}}$ in the earlier stage when $t=\bigo{n}$. In this paper we answer this question in a more general setting: We determine $\numberEdges{\pro{n}{t}}$ in the case $t=\bigo{n}$ for a wide range of graph classes, including outerplanar graphs, planar graphs, and graphs on surfaces (see \Cref{thm:main}).

Gerke, Schlatter, Steger, and Taraz also studied a random graph $\acc{n}{m_0}$, which is a graph obtained from the $\cl$-constrained random graph process when $m_0$ many edges have actually been accepted: In other words,
\begin{align}\label{eq:9}
	\acc{n}{m_0}=\pro{n}{t_0}, \quad \text{where}\quad t_0:=\min\setbuilder{t}{\numberEdges{\pro{n}{t}}=m_0}.
\end{align}
Equivalently, $\acc{n}{m_0}$ can be obtained by the so-called \textit{random greedy process}. There we start with an empty graph on $n$ vertices and in each step we add an edge chosen uniformly at random from those which are not yet in the graph and do not violate the property $\cl$. They showed that in the \lq dense\rq\ regime when $m_0=cn/2$ for $2<c<6$, \whp\ $\acc{n}{m_0}$ is connected.

\begin{thm}[{\cite[Theorem 1.2]{GerkeSchlatterStegerTaraz2008}}]\label{thm:connected}
Let $\cl$ be the class of planar graphs and $\acc{n}{m_0}$ be as defined in \eqref{eq:9}.
If $m_0=m_0(n)$ is such that $m_0=cn/2$ for $2<c<6$, then \whp\ $\acc{n}{m_0}$ is connected.
\end{thm}

In other words, \Cref{thm:connected} says that if $m_0=cn/2$ for $2<c<6$, then \whp\ the largest component $\largestcomponent{\acc{n}{m_0}}$ of $\acc{n}{m_0}$ contains all $n$ vertices. In this paper we determine the asymptotic order of $\largestcomponent{\acc{n}{m_0}}$ also in the \lq sparse\rq\ regime when $c<2$ (see \Cref{thm:largestComponent}).

Another well-known random graph model is the {\em uniform} random planar graph. More generally, let $P(n,m_0)$ denote a graph that is chosen uniformly at random from all graphs in $\cl$ having vertex set $[n]$ and $m_0$ edges. A classical question is whether or not a random graph $\acc{n}{m_0}$ \lq behaves\rq\ like the uniform random graph $P(n,m_0)$. For example, Gim\'{e}nez and Noy \cite{GimenezNoy2009} showed that the probability that the uniform random planar graph $P(n,m_0)$ is connected is bounded away from one if $m_0=cn/2$ for $2<c<6$, i.e., a statement as in \Cref{thm:connected} is not true for $P(n,m_0)$. A consequence of our results will be that the largest components of the two random graphs $\acc{n}{m_0}$ and $P(n,m_0)$ behave \lq differently\rq\ also in the sparse regime when $m_0=cn/2$ for $1<c<2$ (see \eqref{eq:11}).

Another natural property which was considered in \cite{GerkeSchlatterStegerTaraz2008} is the number of edges one has to query until $m_0$ of them have been accepted. We denote this number by $\con{m_0}$, i.e.,
\begin{align}\label{eq:10}
	\con{m_0}:=\min\setbuilder{t}{\numberEdges{\pro{n}{t}}=m_0}.
\end{align}
Gerke, Schlatter, Steger, and Taraz \cite{GerkeSchlatterStegerTaraz2008} asked the order of $\con{3n/4}$ when $\cl$ is the class of planar graphs. In \Cref{coro:considered} we will determine $\con{3n/4}$ as a special case of our general result.

\subsection{Main results}\label{subsec:main}
In the following definition we extract the properties of planar graphs which are essential for our proof, but are satisfied by other well-known classes of graphs that can be characterised by \lq forbidden minors\rq\ (see \Cref{prop:graph_classes}). 

\begin{definition}\label{def:class}
Throughout the paper, let $\cl$ be a class of graphs fulfilling the following properties:
\begin{enumerate}
\item\label{def:classA}
it is not equal to the class of all graphs;
\item\label{def:classB}
it contains all edgeless graphs;
\item\label{def:classC}
it is closed under taking isomorphism;
\item\label{def:classD}
it is closed under taking minors;
\item\label{def:classE}
it is weakly addable, i.e., it is closed under adding an edge between two components;
\item\label{def:classF}
it is closed under adding an edge in a tree component.
\end{enumerate}
\end{definition}

Throughout the paper, we consider only vertex-labelled simple undirected graphs. It is straightforward to check the properties in \Cref{def:class} for the following general class of graphs.
\begin{prop}\label{prop:graph_classes}
For any $r\in\N$ let $H_1, \ldots, H_r$ be 2-edge-connected graphs that contain at least two cycles. Then the class $\mathcal{H}$ of all graphs that contain none of $H_1, \ldots, H_r$ as a minor fulfils the properties \ref{def:classA}--\ref{def:classF} in \Cref{def:class}.
\end{prop}

Prominent examples for the class $\mathcal{H}$ in \Cref{prop:graph_classes} are the following:
\begin{itemize}
	\item
	the class of all cactus graphs ($H_1=\text{\lq diamond graph\rq}$, that is, $K_4$ minus one edge);
	\item
	the class of all outerplanar graphs ($H_1=K_4$, $H_2=K_{2,3}$);
	\item
	the class of all series-parallel graphs ($H_1=K_4$);
	\item
	the class of all planar graphs ($H_1=K_5$, $H_2=K_{3,3}$);
	\item
	the class of all graphs embeddable on an orientable surface of genus $g\in\N$ (only the existence of graphs $H_1, \ldots, H_r$ is known, see \cite{RobertsonSeymour}).
\end{itemize}

To state our main results, we need also the following definition.
\begin{definition}\label{def:func}
Given $c>1$ let $\sol{c}$ be the unique positive solution of the equation $1-x=e^{-cx}$ and define 
\begin{align*}
\func{c}:=2\sol{c}+c\left(1-\sol{c}\right)^2.
\end{align*}
Denote by $f^{-1}$ the inverse function of $f$.
\end{definition}

Note that $\sol{c}$ is equal to the survival probability of a Galton-Watson process with offspring distribution Poisson with mean $c$.
Basic properties of the function $f:\left(1,\infty\right)\to \left(1,2\right)$, including the existence of the inverse function $f^{-1}$, can be found in \Cref{lem:function}.

In the following theorem we provide the asymptotic order of the number of accepted edges $\numberEdges{\pro{n}{t}}$ when $t=\bigo{n}$ for any class of graphs $\cl$ satisfying the properties in \Cref{def:class}. As $\numberEdges{\pro{n}{t}}$ is \lq quite close\rq\ to $t$ in this early stage of the evolution, it is more convenient to state the asymptotic order of the number of {\em rejected} edges $\rej{t}=t-\numberEdges{\pro{n}{t}}$ instead of $\numberEdges{\pro{n}{t}}$.
\begin{thm}\label{thm:main}
	Let $\cl$ be a class of graphs satisfying the properties \ref{def:classA}--\ref{def:classF} in \Cref{def:class} and $\left(\pro{n}{t}\right)_{t=0}^{N}$ be the $\cl$-constrained random graph process. Let $h=h(n)=\smallomega{1}$ be a function which tends to $\infty$ arbitrarily slowly as $n\to\infty$. Let $t=t(n)\in\left[N\right]$ and $s=s(n)$. Then \whp
	\begin{align*}
	\rej{t}=
	\begin{cases}
	0 & \text{if} ~~ t=n/2-s ~~\text{for}~~ s\gg n^{2/3};
	\\
	\bigo{h}& \text{if} ~~ t=n/2+s ~~\text{for}~~ s=\bigo{n^{2/3}};
	\\
	\Th{s^3/n^2}& \text{if} ~~ t=n/2+s ~~\text{for}~~ n^{2/3}\ll s\ll n;
	\\
	\left(c-\func{c}+\smallo{1}\right)n/2& \text{if} ~~ t=cn/2 ~~\text{for}~~ c>1.
	\end{cases}
	\end{align*}
\end{thm}

\begin{figure}
	\includegraphics[scale=0.6]{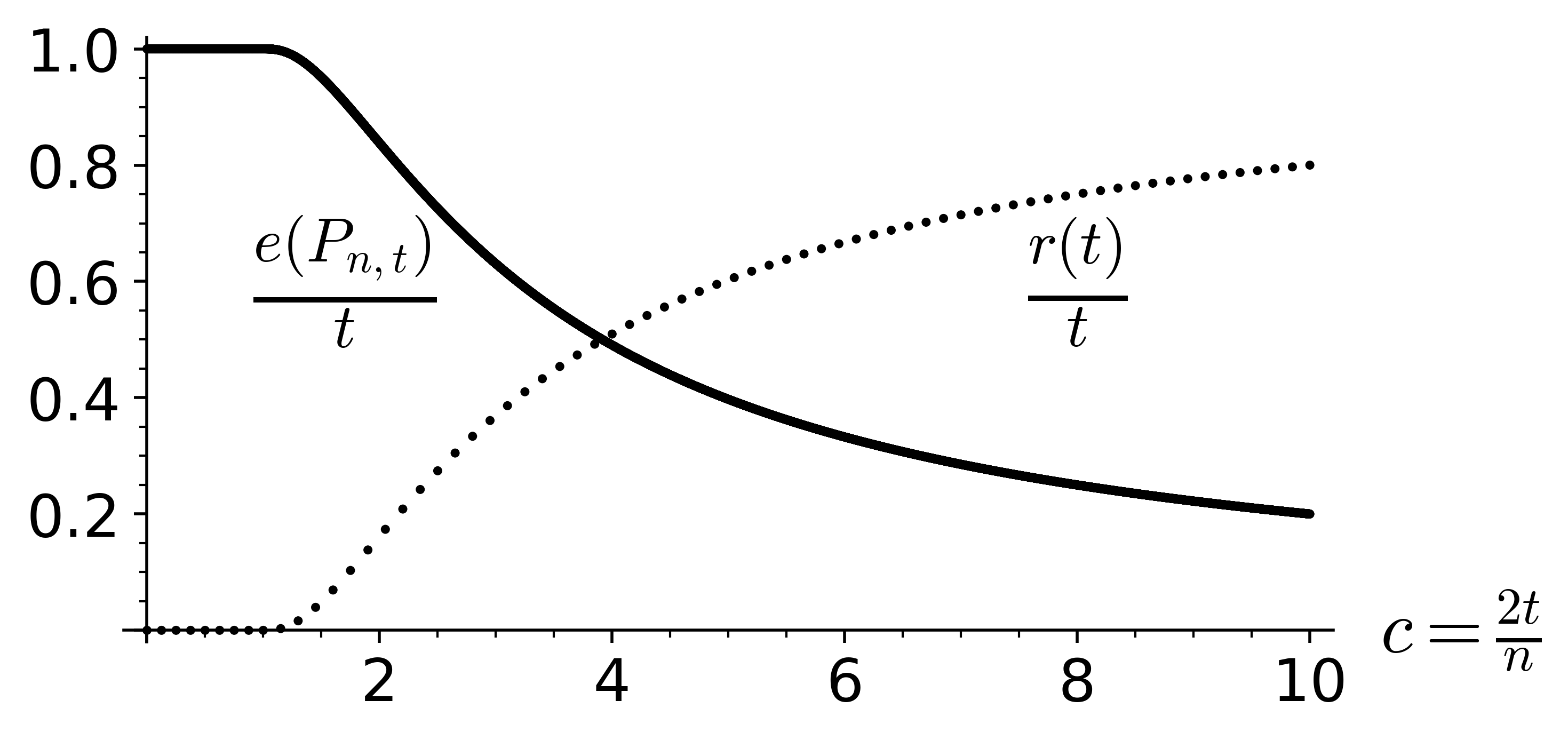}
	\caption{The fraction $\numberEdges{\pro{n}{t}}/t$ of edges that are accepted up to step $t=cn/2$ (solid line) and the fraction $\rej{t}/t$ of rejected edges (dotted line). For $c>1$ we have $\numberEdges{\pro{n}{t}}/t\sim \func{c}/c$ and $\rej{t}/t\sim 1-\func{c}/c$.}
	\label{fig:fraction}
\end{figure}

When $t=n/2+s$ for $s=\bigo{n^{2/3}}$, the statement that \whp\ $\rej{t}=\bigo{h}$ can be equivalently formulated as follows: For each $\lambda>0$ there exists a $C=C(\lambda)$ such that $\prob{\rej{t}<C}>1-\lambda$ for all sufficiently large $n\in\N$, i.e., the statement is \lq slightly weaker\rq\ than having whp $\rej{t}=\bigo{1}$.

In the case $t=cn/2$ for $c>1$ the statement that \whp\ $\rej{t}=\left(c-\func{c}+\smallo{1}\right)n/2$ can be simplified to \whp\ $\numberEdges{\pro{n}{t}}=\left(\func{c}+\smallo{1}\right)n/2$. Using \Cref{thm:main} we obtain the following nice, alternative description of a $\cl$-constrained random graph process: In the very early stage of the process when $t\leq n/2+\smallo{n}$, \lq almost\rq\ all edges are accepted. More formally, \whp\ the \lq next\rq\ edge $e_{t+1}$ will be accepted as long as $t\leq n/2+\smallo{n}$. However, this changes when $t=cn/2$ for $c>1$: The acceptance mainly depends whether or not both of the endpoints of $e_{t+1}$ lie in the largest component of $\pro{n}{t}$.

\begin{coro}\label{coro:main}
Let $\cl$ be a class of graphs satisfying the properties \ref{def:classA}--\ref{def:classF} in \Cref{def:class} and $\left(\pro{n}{t}\right)_{t=0}^{N}$ be the $\cl$-constrained random graph process. Let $t=t(n)\in\left[N\right]$ be such that $t=cn/2$ for a constant $c>1$, and let $L=\largestcomponent{\pro{n}{t}}$ denote the largest component of $\pro{n}{t}$. Then \whp\ the following hold.
\begin{enumerate}
\item\label{coro:mainA}
If $e_{t+1}\subseteq \vertexSet{\Largestcomponent}$, then $e_{t+1}$ is rejected.
\item\label{coro:mainB}
If $e_{t+1}\nsubseteq \vertexSet{\Largestcomponent}$, then $e_{t+1}$ is accepted.
\end{enumerate}
\end{coro}

Using \Cref{thm:main} we can determine the asymptotic number of queried edges $\con{m_0}$ until $m_0$ of them have been accepted in the case $m_0=cn/2$ for $1<c<2$. In particular, this answers the open problem on $\con{3n/4}$ for the property $\cl$ of being planar from \cite{GerkeSchlatterStegerTaraz2008}.

\begin{coro}\label{coro:considered}
	Let $\cl$ be a class of graphs satisfying the properties \ref{def:classA}--\ref{def:classF} in \Cref{def:class} and $\con{m_0}$ be as defined in \eqref{eq:10}. If $m_0=cn/2$ for a constant $c\in(1,2)$, then \whp
	\begin{align*}
		\con{m_0}=\left(\invFunc{c}+\smallo{1}\right)n/2.
	\end{align*}
	In particular, \whp\ $\con{3n/4}=\left(\invFunc{3/2}+\smallo{1}\right)n/2$, where $\invFunc{3/2}=1.6188\ldots$.
\end{coro}

We note that \Cref{coro:considered} follows directly from \Cref{thm:main} and the observation that $f$ is strictly increasing (see \Cref{lem:function}\ref{lem:functionB}).

Our next main result provides the asymptotic order of the largest component of $\acc{n}{m_0}$.

\begin{thm}\label{thm:largestComponent}
	Let $\cl$ be a class of graphs satisfying the properties \ref{def:classA}--\ref{def:classF} in \Cref{def:class} and $\left(\pro{n}{t}\right)_{t=0}^{N}$ be the $\cl$-constrained random graph process. Let $m_0=m_0(n)\in\left[N\right]$ and $s=s(n)$. Let $\acc{n}{m_0}$ be defined as in \eqref{eq:9} and let $\numberVertices{\largestcomponent{\acc{n}{m_0}}}$ denote the number of vertices in the largest component of $\acc{n}{m_0}$. Then \whp
	\begin{align*}
		\numberVertices{\largestcomponent{\acc{n}{m_0}}}=
		\begin{cases}
			\bigo{\log n} & \text{if} ~~ m_0=cn/2 ~~\text{for}~~ c<1;
			\\
			\left(1/2+\smallo{1}\right)n^2/s^2\log\left(s^3/n^2\right) & \text{if} ~~ m_0=n/2-s ~~\text{for}~~ n^{2/3}\ll s\ll n;
			\\
			\Th{n^{2/3}}& \text{if} ~~ m_0=n/2+s ~~\text{for}~~ s=\bigo{n^{2/3}};
			\\
			\left(4+\smallo{1}\right)s& \text{if} ~~ m_0=n/2+s ~~\text{for}~~ n^{2/3}\ll s\ll n;
			\\
			\left(\sol{\invFunc{c}}+\smallo{1}\right)n& \text{if} ~~ m_0=cn/2 ~~\text{for}~~ 1<c<2;
			\\
			\left(1+\smallo{1}\right)n& \text{if} ~~ m_0=cn/2 ~~\text{for}~~ c=2;
			\\
			n& \text{if} ~~ m_0=cn/2 ~~\text{for}~~ c>2.
		\end{cases}
	\end{align*}
\end{thm}

For the property $\cl$ of being planar \Cref{thm:largestComponent} reveals a different behaviour of $\acc{n}{m_0}$ in the \lq sparse\rq\ regime than that of the uniform random planar graph $P(n,m_0)$. More formally, if $m_0=cn/2$ for $1<c<2$, then \whp
\begin{align}\label{eq:11}
\numberVertices{\largestcomponent{\acc{n}{m_0}}}=\left(\sol{\invFunc{c}}+\smallo{1}\right)n~~\quad>~~\quad\left(c-1+\smallo{1}\right)n=\numberVertices{\largestcomponent{P(n,m_0)}},
\end{align}
where the last equality follows from \cite{KangLuczak2012}. We refer to \Cref{fig:cv} for an illustration of $\numberVertices{\largestcomponent{\acc{n}{m_0}}}$ and $\numberVertices{\largestcomponent{P(n,m_0)}}$.

\begin{figure}
	\includegraphics[scale=0.535]{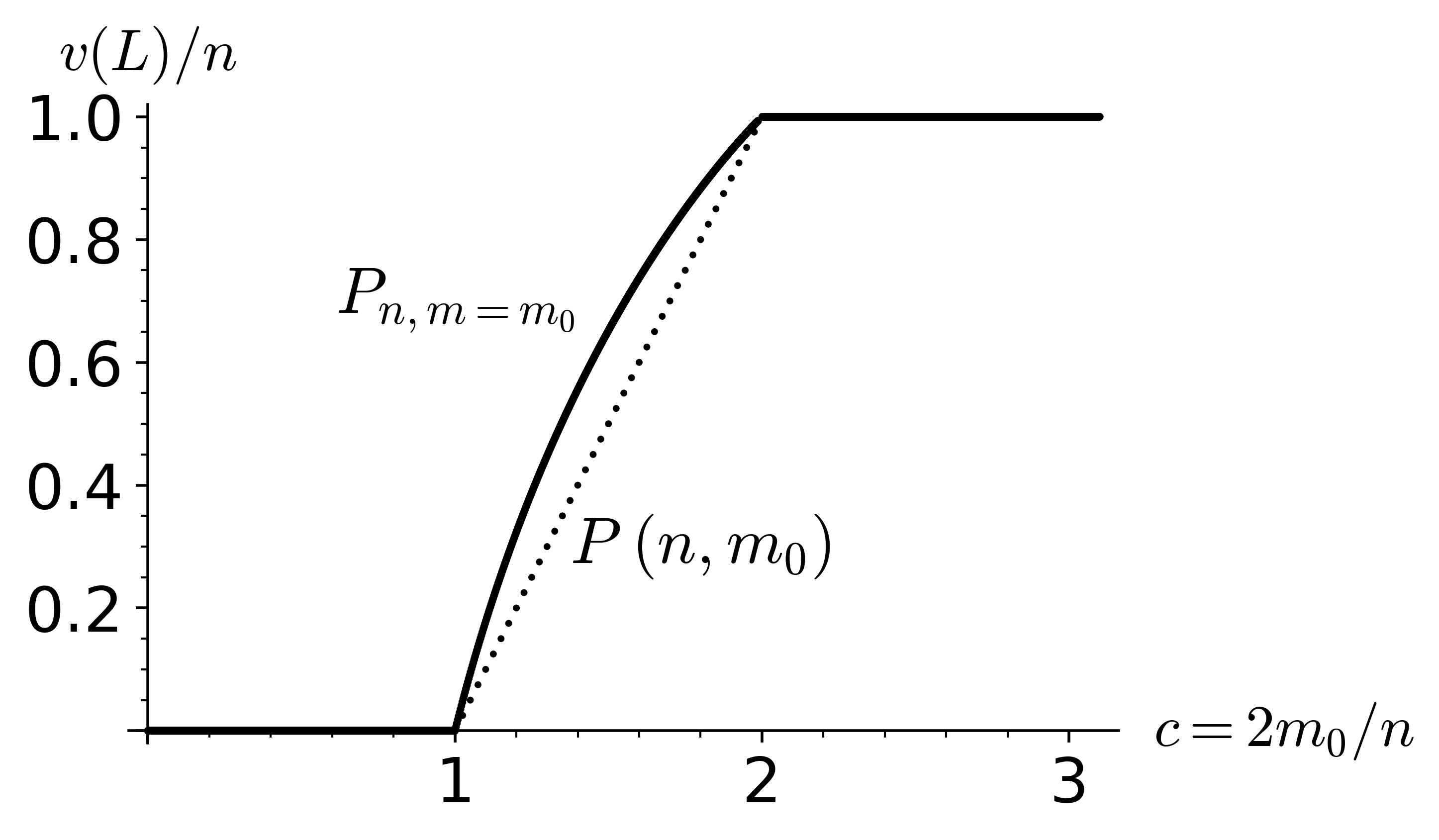}
	\includegraphics[scale=0.465]{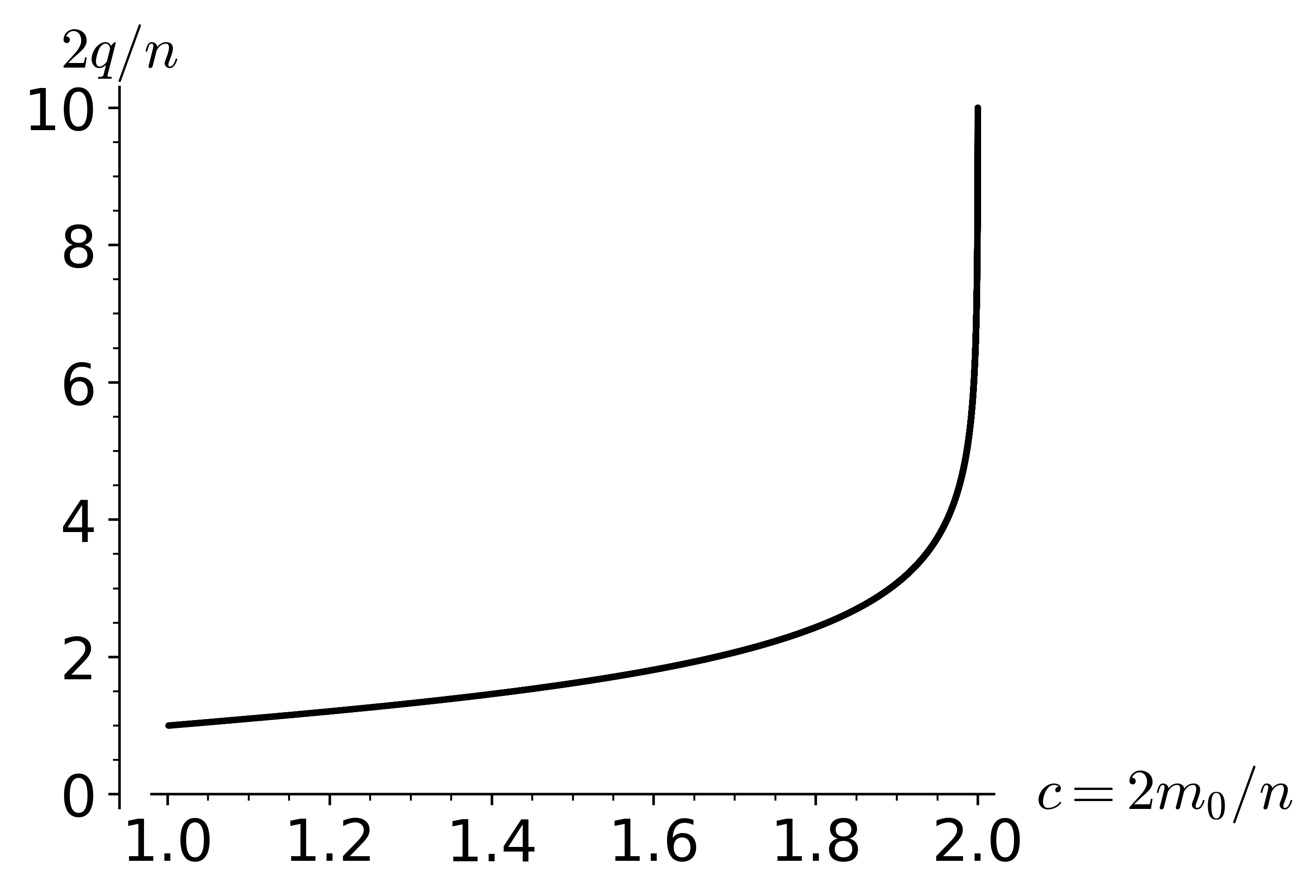}
	\caption{On the left-hand side: the fraction of vertices lying in the largest component of $\acc{n}{m_0}$ (solid line) and in the uniform random planar graph $P(n,m_0)$ (dotted line) are shown. For $1<c<2$ we have $\numberVertices{\largestcomponent{\acc{n}{m_0}}}/n\sim \sol{\invFunc{c}}$ and $\numberVertices{\largestcomponent{P(n,m_0)}}/n\sim c-1$. On the right-hand side: $2\con{m_0}/n\sim f^{-1}(c)$ is illustrated for the case that $m_0=cn/2$ for $c\in(1,2)$. }
	\label{fig:cv}
\end{figure}

\subsection{Outline of the paper}
The rest of the paper is structured as follows. After providing the necessary definitions and concepts in \Cref{sec:preliminaries}, we prove \Cref{thm:main} in \Cref{sec:proof_main}. In \Cref{sec:proof_coro} we use so-called \lq addable\rq\ and \lq forbidden\rq\ edges to prove \Cref{coro:main}. Finally in \Cref{sec:accepted_graph}, we prove \Cref{thm:largestComponent}.

\section{Preliminaries}\label{sec:preliminaries}
\subsection{Notations for graphs}
We begin with some notations for graphs that will be used in the rest of the paper.
\begin{definition}
	Given a graph $H$ we denote by
	\begin{itemize}
		\item 
		$\vertexSet{H}$ the vertex set of $H$ and
		\item[]
		$\numberVertices{H}$ the order of $H$, i.e., the number of vertices in $H$;
		\item 
		$\edgeSet{H}$ the edge set of $H$ and
		\item[]
		$\numberEdges{H}$ the size of $H$, i.e., the number of edges in $H$;	
		\item
		$\maxdegree{H}$ the maximum degree of $H$;
		\item 
		$\largestcomponent{H}$ the largest component of $H$;
		\item
		$\core{H}$ the 2-core of $H$, which is the maximal subgraph of $H$ with minimum degree at least two;
		\item
		$\ex{H}:=\numberEdges{H}-\numberVertices{H}+\nt{H}$ the excess of $H$, where $\nt{H}$ is the number of tree components in $H$.
	\end{itemize}
\end{definition}
\begin{definition}\label{def:graph_class}
	Given a class $\cl$ of graphs, we write
	$\cl(n)$ for the subclass of $\cl$ containing the graphs on vertex set $[n]$ 
	and $\cl(n,m)$ for the subclass of $\cl$ containing the graphs on vertex set $[n]$ with $m$ edges, respectively.
\end{definition}

\subsection{Properties of the \ER\ random graph}
In this section we state the properties of the \ER\ random graph $\er{n}{t}$ which we will use in our proofs. First we consider the case $t=n/2+s$ for $n^{2/3}\ll s\ll n$ and then the case $t=cn/2$ for $c>1$.
\begin{thm}[\cite{Bollobas1982,Luczak1990,Luczak1991,Pavlov1977}]\label{thm:er}
	Let $t=t(n)\in\left[N\right]$ and $s=s(n)$ be such that $t=n/2+s$ for $n^{2/3}\ll s\ll n$ and $\Largestcomponent=\largestcomponent{\er{n}{t}}$ be the largest component of the \ER\ random graph $\er{n}{t}$. Furthermore, let $C=\core{L}$ be the 2-core of $\Largestcomponent$ and $F$ be the forest obtained from $L$ by deleting the edges of $C$. Then \whp
	\begin{enumerate}
		\item\label{thm:erA}
		$\numberVertices{\Largestcomponent}=\left(4+\smallo{1}\right)s$;
		\item\label{thm:erE}
		$\ex{\er{n}{t}}=\Th{s^3/n^2}$;
		\item\label{thm:erB}
		each tree component of $F$ is of order $\smallo{s}$;
		\item\label{thm:erC}
		$\maxdegree{\er{n}{t}}=\left(1+\smallo{1}\right)\log n/\log\log n$;
		\item\label{thm:erD}
		$\maxdegree{C}=3$ if in addition $s\ll n^{3/4}$.
	\end{enumerate}
\end{thm}

We note that \ref{thm:erA} and \ref{thm:erE} are shown in \cite{Luczak1990}, \ref{thm:erC} in \cite{Bollobas1982}, and \ref{thm:erD} in \cite{Luczak1991}, respectively. Furthermore, \ref{thm:erB} follows by the fact from \cite{Pavlov1977} that conditioned on fixed values of $\numberVertices{\Largestcomponent}$ and $\numberVertices{C}$, \whp\ all tree components of $F$ are of order $\smallo{\numberVertices{\Largestcomponent}}$ as long as $\numberVertices{C}^2=\smallomega{\numberVertices{\Largestcomponent}}$. Furthermore, \whp\ $\er{n}{t}$ satisfies this condition, because \whp\ $\numberVertices{C}=\Th{s^2/n}$ by \cite{Luczak1991} and $\numberVertices{\Largestcomponent}=\left(4+\smallo{1}\right)s$ by \ref{thm:erA}.

Next we collect some properties of the \ER\ random graph $\er{n}{t}$ when $t=cn/2$ for $c>1$.

\begin{thm}[\cite{ErdoesRenyi1960}]\label{thm:erInt}
	Let $t=t(n)\in\left[N\right]$ be such that $t=cn/2$ for $c>1$. Let $G=\er{n}{t}$ be the \ER\ random graph, and $\Largestcomponent=\largestcomponent{G}$ the largest component of $G$. Furthermore, let $\sol{c}$ and $\func{c}$ be as in \Cref{def:func}. Then \whp
	\begin{enumerate}
		\item\label{thm:erIntA}
		$\numberVertices{\Largestcomponent}=\left(\sol{c}+\smallo{1}\right)n$;
		\item\label{thm:erIntC}
		all components of $G$ apart from $L$ are of order $\smallo{n}$;
		\item\label{thm:erIntB}
		$\ex{G}=\left(c-\func{c}+\smallo{1}\right)n/2$.
	\end{enumerate}
\end{thm}

\subsection{Properties of $\pro{n}{t}$ and $\cl$}
We will often use the following simple observation.
\begin{remark}\label{rem:component_structure}
Due to properties \ref{def:classB} and \ref{def:classE} of \Cref{def:class} there is a path between two vertices in $\pro{n}{t}$ if and only if there is one in $\er{n}{t}$.
\end{remark}

Next, we show that the number of rejected edges up to step $t$ is bounded above by the excess of $\er{n}{t}$. This will be a main ingredient to obtain the upper bounds in \Cref{thm:main}.

\begin{lem}\label{lem:upperBound_excess}
For all $t\in \left[N\right]$ we have
\begin{align*}
\rej{t}\leq \ex{\er{n}{t}}.
\end{align*}
\end{lem}

\begin{proof}
We consider an edge $e_i$ which is rejected in the $\cl$-constrained random graph process. By \Cref{def:class}\ref{def:classE} and \ref{def:classF} the two endpoints of $e_i$ lie in the same component of $\pro{n}{i-1}$, which is not a tree component. Together with the fact $\pro{n}{i-1}\subseteq \er{n}{i-1}$ it implies that adding $e_i$ to $\er{n}{i-1}$ increases the excess by one, i.e., $\ex{\er{n}{i}}=\ex{\er{n}{i-1}}+1$. As $\ex{\er{n}{t}}$ is non-decreasing in $t$, this implies the statement.
\end{proof}

A graph class $\mathcal{A}$ for which there exists a constant $c>0$ such that $\left|\mathcal{A}(n)\right|\leq n!c^n$ for all $n\in\N$ is often called \textit{small} (see e.g., \cite{NorineSeymourThomasWollan2006}). The following statement shows that the class $\cl$ in \Cref{def:class} is small.
\begin{thm}[\cite{NorineSeymourThomasWollan2006}]\label{thm:small_classes}
	Let $\cl$ be a class of graphs satisfying the properties \ref{def:classA}, \ref{def:classC}, and \ref{def:classD} in \Cref{def:class}. Then there exists a constant $c>0$ such that $\left|\cl(n)\right|\leq n!c^n$ for all $n\in\N$. 
\end{thm}

\subsection{Decomposition of graphs}
In the proof of \Cref{thm:main} we will split the largest component of $\pro{n}{t}$ into connected parts of roughly equal size. To that end, we will use the following lemma, which is an extension of \cite[Proposition 4.5]{KrivelevichNachmias2006} to vertex-weighted graphs.
\begin{lem}\label{lem:decomposition}
Let $H$ be a connected graph with maximum degree at most $\Delta\geq 1$. We assign each vertex $x\in\vertexSet{H}$ a vertex-weight $\weight{x}>0$. Assume that $\max_{x\in\vertexSet{H}}\weight{x}\leq M$ for some $M>0$. Then, given $a>0$ there exist disjoint vertex sets $V_1, \ldots, V_r\subseteq \vertexSet{H}$ such that
\begin{itemize}
\item
$H\left[V_i\right]$ is connected for each $i\in[r]$;
\item
$a\leq W_i\leq a \Delta+M$ for each $i\in[r]$;
\item
$W(H)-\sum_{i=1}^{r}W_i<a$,
\end{itemize}
where $W(H):=\sum_{x\in V(H)}\weight{x}$ denotes the total vertex-weight of $H$ and $W_i:=\sum_{x\in V_i}\weight{x}$ the total vertex-weight of $V_i$, respectively.
\end{lem}

\begin{proof}
We proceed by induction on $\numberVertices{H}$. For $\numberVertices{H}=1$ let $x$ denote the single vertex in $H$. We set $V_1=\left\{x\right\}$ if $\weight{x}\geq a$ and let $r=0$ otherwise.
 
Now assume $\numberVertices{H}>1$. If $W(H)<a$, we set $r=0$. Assume otherwise that $W(H)\geq a$. We perform a breadth-first search (BFS) starting at some arbitrary vertex $y$. For each vertex $x\in\vertexSet{H}$ let $D_x$ be the set of vertices consisting of $x$ and all descendants of $x$ in the BFS-tree and let $W(x):=\sum_{u\in D_x}\weight{u}$ be the total vertex-weight of $D_x$. Let $z\in \vertexSet{H}$ be a vertex such that $W(z)$ is minimal among all vertices $x$ with $W(x)\geq a$. We note that such a vertex exists, as $W(y)=W(H)\geq a$. Let $z_1, \ldots, z_k$ be the neighbours of $z$ that are contained in $D_z$. By minimality of $z$ and the fact $W(z_i)<W(z)$ we have $W(z_i)<a$ for all $i\in[k]$. Thus, we get $W(z)=\sum_{i=1}^{k}\weight{z_i}+\weight{z}\leq a\Delta +M$. Therefore, we choose $V_1=D_z$. By construction $H\left[V_1\right]$ is connected and we have $a\leq W_1\leq a \Delta+M$. Furthermore, the graph obtained from $H$ by deleting $V_1$ is connected and has less vertices than $H$. Hence, we can apply the induction hypothesis to this graph to obtain the remaining sets $V_2, \ldots, V_r$ of our desired decomposition.
\end{proof}

We will show that if we have a \lq suitable\rq\ decomposition of $\pro{n}{t}$, we cannot add too \lq many\rq\ further edges without creating a minor of the complete graph $K_\ell$ for an appropriate $\ell\in\N$. To that end, we will use the following lemma, which is a \lq weighted\rq\ version of the well-known Tur{\'a}n's theorem.
\begin{lem}\label{lem:weighted_turan}
For fixed $n\in\N$ we assign a vertex-weight $\weight{x}>0$ to each vertex $x$ of $K_n$. Furthermore, we define the edge-weight $\weight{xy}:=\weight{x}\weight{y}$ for each edge $xy\in\edgeSet{K_n}$. Let $S_n:=\sum_{x\in\vertexSet{K_n}}\weight{x}$ be the total vertex-weight of $K_n$ and $M_n:=\max_{x\in\vertexSet{K_n}}\weight{x}$ be the maximum vertex-weight of $K_n$. For each subgraph $H\subseteq K_n$ denote by $\weight{H}:=\sum_{e\in \edgeSet{H}}\weight{e}$ the total edge-weight of $H$. Then for each $\ell\geq 2$ we have 
\begin{align*}
\weight{K_n}-\max\setbuilder{\weight{H}}{H\subseteq K_n, K_\ell\nsubseteq H}\geq \frac{S_n}{2}\left(\frac{S_n}{\ell-1}-M_n\right).
\end{align*}
\end{lem}
\begin{proof}
To ease notation, let $A_n:=\max\setbuilder{\weight{H}}{H\subseteq K_n, K_\ell\nsubseteq H}$. In \cite{BennettEnglishTalanda-Fisher2019} Bennett, English, and Talanda-Fisher showed that
\begin{align*}
A_n=\frac{1}{2} \cdot\max_{\mathcal{Q}}\sum_{Q\neq Q'\in\mathcal{Q}}W(Q)W\left(Q'\right),
\end{align*}
where the maximum is taken over all partitions $\mathcal{Q}$ of $\vertexSet{K_n}$ into $\ell-1$ parts and $W(Q):=\sum_{x\in Q}\weight{x}$ denotes the total vertex-weight of $Q\in\mathcal{Q}$. Now let $\mathcal{Q}^\ast$ be some partition for which the maximum is attained. Note that $\sum_{Q\in\mathcal{Q}^\ast}W(Q)=\sum_{x\in\vertexSet{K_n}}\weight{x}=S_n$. Furthermore, we have
\begin{align*}
A_n~=~\frac{1}{2}\sum_{Q\neq Q'\in\mathcal{Q}^\ast}W(Q)W\left(Q'\right)~=~\frac{1}{2} \left(\sum_{Q\in\mathcal{Q}^\ast}W(Q)\right)^2-\frac{1}{2}\sum_{Q\in\mathcal{Q}^\ast}W(Q)^2~=~\frac{S_n^2}{2}-\frac{1}{2}\sum_{Q\in\mathcal{Q}^\ast}W(Q)^2.
\end{align*}
Using the \lq AM-QM inequality\rq, in other words, $\sum_{Q\in\mathcal{Q}^\ast}W(Q)^2\geq \frac{1}{\ell-1}\left(\sum_{Q\in\mathcal{Q}^\ast}W(Q)\right)^2$, we obtain
\begin{align}\label{eq:2}
A_n~\leq~\frac{S_n^2}{2}-\frac{1}{2}\cdot \frac{1}{\ell-1}\left(\sum_{Q\in\mathcal{Q}^\ast}W(Q)\right)^2~=~ \frac{S_n^2}{2}\left(1-\frac{1}{\ell-1}\right).
\end{align}
Finally, the total edge-weight of $K_n$ satisfies
\begin{align*}
	\weight{K_n}~=~\sum_{e\in\edgeSet{K_n}}\weight{e}~=~\frac{1}{2}\sum_{x\neq y\in\vertexSet{K_n}}\weight{x}\weight{y}~=~\frac{S_n^2}{2}-\frac{1}{2}\sum_{x\in\vertexSet{K_n}}\weight{x}^2~\geq~ \frac{S_n^2}{2}-\frac{S_nM_n}{2}.
\end{align*}
This together with \eqref{eq:2} implies the statement.
\end{proof}

\section{Rejected edges: proof of \Cref{thm:main}}\label{sec:proof_main}
The upper bounds follow immediately from \Cref{lem:upperBound_excess}.

\subsection{Proof of upper bounds}\label{subsec:proofUpper}
Due to \cite{Britikov1989,JansonKnuthLuczakPittel1993} and \Cref{thm:er}\ref{thm:erE} and \ref{thm:erInt}\ref{thm:erIntB} we have that \whp
\begin{align*}
	\ex{\er{n}{t}}=
	\begin{cases}
		0 & \text{if} ~~ t=n/2-s ~~\text{for}~~ s\gg n^{2/3};
		\\
		\bigo{h}& \text{if} ~~ t=n/2+s ~~\text{for}~~ s=\bigo{n^{2/3}};
		\\
		\Th{s^3/n^2}& \text{if} ~~ t=n/2+s ~~\text{for}~~ n^{2/3}\ll s\ll n;
		\\
		\left(c-\func{c}+\smallo{1}\right)n/2& \text{if} ~~ t=cn/2 ~~\text{for}~~ c>1.
	\end{cases}
\end{align*}
Together with \Cref{lem:upperBound_excess} this implies the upper bounds in \Cref{thm:main}. 

Before proving the lower bounds we first sketch the main ideas.

\subsection{Proof idea for lower bounds}\label{subsec:proofidea}
First we consider the case $t=n/2+s$ for $n^{2/3}\ll s\ll n$. We use a consequence of the properties \ref{def:classA} and \ref{def:classD} in \Cref{def:class} that there exists an $\ell\in\N$ such that no graph in $\cl$ contains the complete graph $K_\ell$ as a minor. We then apply a \lq sprinkling\rq\ type argument: Let $P=\pro{n}{t}$ and let $P'=\pro{n}{t'}$ for $t'=n/2+s/2$. We first reveal the edges $e_1, \ldots, e_{t'}$. Given the realisation of $P'$, we split the vertex set $\vertexSet{\largestcomponent{P'}}$ of the largest component of $P'$ into disjoint sets $V_1, \ldots, V_r$ of \lq almost\rq\ equal sizes such that $P'\left[V_i\right]$ is connected for each $i\in[r]$, where $r\geq \ell$. Next we reveal the remaining edges $e_{t'+1}, \ldots, e_t$ up to step $t$ and show that for each pair $i\neq j$ there are \lq many\rq\ edges between $V_i$ and $V_j$ which are queried up to step $t$. As $K_\ell$ is not a minor of $P'$, there are some pairs for which all of these edge are rejected. This provides a lower bound on the number of rejected edges. The precise way of decomposing the largest component of $P'$ differs in the cases $s\ll n^{3/4}$ and $s\gg n^{17/24}$. We note that it is sufficient to deal with these two cases, because the general case $n^{2/3}\ll s\ll n$ follows by considering appropriate subsequences.

The starting point for the case $t=cn/2$ for $c>1$ is \Cref{thm:small_classes}. Roughly speaking, it says that only a very small number of all graphs on $n$ vertices lie in $\cl(n)$. Using that we show that for each $\delta>0$, \whp\ there is no graph $H\in\cl\left(n, \left(1+\delta\right)n\right)$ such that all edges of $H$ are already queried before step $\bar{t}:=n \cdot \log n$. In particular, this shows that \whp\ $\numberEdges{\pro{n}{\bar{t}}}\leq \left(1+\smallo{1}\right)n$. It is well known that \whp\ $\er{n}{\bar{t}}$ and therefore also $\pro{n}{\bar{t}}$ are connected. Thus, we obtain that \whp\ $\ex{\pro{n}{t}}\leq \ex{\pro{n}{\bar{t}}}=\smallo{n}$. Furthermore, we have that
\begin{align*}
t-\numberEdges{\pro{n}{t}}=\numberEdges{\er{n}{t}}-\numberEdges{\pro{n}{t}}\geq \ex{\er{n}{t}}-\ex{\pro{n}{t}}.
\end{align*}
Hence, we get a lower bound by using \whp\ $\ex{\pro{n}{t}}=\smallo{n}$ and $\ex{\er{n}{t}}=\left(c-\func{c}+\smallo{1}\right)n/2$ from \Cref{thm:erInt}\ref{thm:erIntB}.

\subsection{Proof of lower bounds}\label{subsec:proofLower}
(i) We start with the case $t=n/2+s$ for $n^{2/3}\ll s\ll n$. 

Take $t'=n/2+s/2$ and let $P'=\pro{n}{t'}$, $G'=\er{n}{t'}$, $P=\pro{n}{t}$, and $G=\er{n}{t}$. Furthermore, let $\largestcomponent{P'}$ and $\largestcomponent{G'}$ be the largest components of $P'$ and $G'$, respectively. 

Due to the properties \ref{def:classA} and \ref{def:classD} in \Cref{def:class}, there exists an $\ell\in\N$ such that there is no graph in $\cl$ having the complete graph $K_\ell$ as a minor. Now we distinguish two cases.

\textit{Case 1: $s\ll n^{3/4}$.}
First reveal the edges $e_1, \ldots, e_{t'}$.

Let $C=\core{\largestcomponent{G'}}$ be the 2-core of $\largestcomponent{G'}$ and $\forest{G'}$ the forest obtained from $\largestcomponent{G'}$ by deleting the edges of $C$. Moreover, for a vertex $x\in \vertexSet{C}$ let $T_x$ be the tree component of $\forest{G'}$ containing $x$. By \Cref{def:class}\ref{def:classB} and \ref{def:classE} we have $\vertexSet{\largestcomponent{P'}}=\vertexSet{\largestcomponent{G'}}$ and $\edgeSet{\largestcomponent{P'}}\subseteq\edgeSet{\largestcomponent{G'}}$; furthermore, each edge of $\forest{G'}$ is also contained in $\largestcomponent{P'}$. Thus, there is a connected and spanning subgraph $C'\subseteq C$ such that $\largestcomponent{P'}$ can be obtained by replacing each vertex $x$ in $C'$ by the tree $T_x$. 

We apply \Cref{lem:decomposition} to $C'$, where we define the vertex-weight of a vertex $x\in\vertexSet{C'}$ by $\weight{x}:=\left|\vertexSet{T_x}\right|$. Then due to \Cref{thm:er}\ref{thm:erA}, \ref{thm:erB}, and \ref{thm:erD}, \whp\ the total vertex-weight of $C'$ satisfies
\begin{align*}
W\left(C'\right):=\sum_{x\in\vertexSet{C'}}\weight{x}=\sum_{x\in\vertexSet{C'}}\left|\vertexSet{T_x}\right|=\numberVertices{\largestcomponent{P'}}=\numberVertices{\largestcomponent{G'}}=(2+\smallo{1})s,
\end{align*}
the maximum vertex-weight of $C'$ satisfies
\begin{align*}
\max_{x\in\vertexSet{C'}}\weight{x}=\max_{x\in\vertexSet{C'}}\left|\vertexSet{T_x}\right|=\smallo{s},
\end{align*}
and the maximum degree of $C'$ is bounded by $\maxdegree{C'}\leq \maxdegree{C}=3$. Assuming this \whp\ event holds, we apply \Cref{lem:decomposition} to $C'$ with $\Delta=3$ and $a=M=s/(3\ell)$ and obtain disjoint vertex sets $\tilde{V}_1, \ldots, \tilde{V}_r\subseteq \vertexSet{C'}$ such that
\begin{itemize}
\item
$C'\left[\tilde{V}_i\right]$ is connected for each $i\in[r]$;
\item
$a=s/(3\ell)\leq W_i\leq a\Delta+M=4s/(3\ell)$ for each $i\in[r]$;
\item
$W\left(C'\right)-\sum_{i=1}^{r}W_i<a=s/(3\ell)$,
\end{itemize}
where $W_i:=\sum_{x\in \tilde{V}_i} \weight{x}=\sum_{x\in\tilde{V}_i}\left|\vertexSet{T_x}\right|$. 

For $i\in[r]$ let $V_i$ be the set of vertices that lie in some $T_x$ for a $x\in\tilde{V}_i$. Then $V_1, \ldots, V_r\subseteq\largestcomponent{P'}$ are pairwise disjoint and satisfy the following properties: 
\begin{itemize}
\item
$P'\left[V_i\right]$ is connected for each $i\in[r]$;
\item
$s/(3\ell)\leq |V_i|\leq 4s/(3\ell)$ for each $i\in[r]$;
\item
$\sum_{i=1}^{r}|V_i|\geq \numberVertices{\largestcomponent{P'}}-s/(3\ell)\geq 4s/3$.
\end{itemize}
Note that $\ell\leq r=\Th{1}$. 
Next we reveal the edges $e_{t'+1}, \ldots, e_t$. We claim that \whp\ for each pair $i\neq j\in[r]$ there are at least $A:=s^3/\left(18\ell^2 n^2\right)$ many edges between points in $V_i$ and $V_j$ which have been queried up to step $t$. Assume that for some $i\neq j\in[r]$ this is not the case and let $k\in\left[N\right]$ be such that $t'<k\leq t$. Then in the step of revealing $e_k$ there were at least $|V_i||V_j|-A\geq s^2/\left(10\ell^2\right)$ many edges going between $V_i$ and $V_j$ which had not been queried yet. Hence, we have 
\begin{align*}
\prob{e_k \text{ has one endpoint in }V_i \text{ and one in }V_j}\geq s^2/\left(5\ell^2n^2\right)=:p.
\end{align*}
Letting $X$ be the number of edges going between $V_i$ and $V_j$ that have been queried up to step $t$ it implies
\begin{align}\label{eq:1}
\prob{X\leq A}\leq \prob{\bin{s/2}{p}\leq A}=\bigo{n^2/s^3}=\smallo{1}.
\end{align}
As $r=\Th{1}$ this shows the claim. By the choice of $\ell$ we know that $K_\ell$ is not a minor of $P=\pro{n}{t}$. Hence, there is a pair $i\neq j\in[r]$ such that there is no edge in $P$ going between $V_i$ and $V_j$. Together with the claim this yields that \whp\ at least $A=\Th{s^3/n^2}$ many edges have been rejected up to step $t$. This concludes the case $s\ll n^{3/4}$.

\textit{Case 2: $s\gg n^{17/24}$.}
First we reveal again only the edges $e_1, \ldots, e_{t'}$.

Using \Cref{thm:er}\ref{thm:erA} and \ref{thm:erC} we have that \whp\ $\numberVertices{\largestcomponent{P'}}=\numberVertices{\largestcomponent{G'}}=\left(2+\smallo{1}\right)s$ and $\maxdegree{\largestcomponent{P'}}\leq \maxdegree{G'}=\left(1+\smallo{1}\right)\log n/\log \log n$. Assuming this \whp\ event holds, we apply \Cref{lem:decomposition} to $\largestcomponent{P'}$ where we assign a vertex-weight $\weight{x}=1$ to each $x\in\vertexSet{\largestcomponent{P'}}$, $\Delta=\log n$, $M=1$, and $a=s/(\ell \log n)$. This leads to disjoint sets $V_1, \ldots, V_r \subseteq \vertexSet{\largestcomponent{P'}}$ such that $\largestcomponent{P'}\left[V_i\right]$ is connected for each $i\in[r]$, 
\begin{align}
s/(\ell \log n)&\leq |V_i|\leq s/\ell+1 \quad \text{for each }i\in[r], \quad \text{and}\label{eq:3}\\
\sum_{i\in[r]}|V_i|&\geq \largestcomponent{P'}-s/(\ell \log n)=(2+\smallo{1})s. \label{eq:4}
\end{align}

Next reveal the remaining edges $e_{t'+1}, \ldots, e_t$. We claim that \whp\ for each pair $i\neq j\in[r]$ there are at least $B=B(i,j):=|V_i||V_j| s/\left(2n^2\right)$ edges that have been queried up to step $t$. To prove the claim, let $i\neq j\in[r]$ be fixed and denote by $X$ the number of edges between $V_i$ and $V_j$ that have been queried up to step $t$. Analogous to \eqref{eq:1} we obtain, with $q:=3|V_i||V_j|/\left(2n^2\right)$,
\begin{align*}
\prob{X\leq B}\leq \prob{\bin{s/2}{q}\leq B}=\bigo{\frac{n^2}{|V_i||V_j|s}}=\bigo{\frac{n^2 \log n^2}{s^3}}=\smallo{n^{-1/9}},
\end{align*}
where we used Chebyshev's inequality, \eqref{eq:3}, and $s\gg n^{17/24}$. As $r=\bigo{\log n}$, the claim follows by the union bound. Next, let $H$ be the graph with (super)vertex set $\vertexSet{H}=\left\{V_1, \ldots, V_r\right\}$ and two vertices $V_i$ and $V_j$ are connected if and only if there is an edge in $P$ going between $V_i$ and $V_j$. We assign each vertex $V_i$ in $H$ the vertex-weight $w(V_i):=|V_i|$. Due to \eqref{eq:3} and \eqref{eq:4} we have that the maximum vertex-weight and the total vertex-weight of $\vertexSet{H}$ satisfy 
\begin{align*}
M_r:=\max_{i\in[r]}\weight{V_i}\leq s/\ell+1 \quad \text{ and } \quad S_r:=\sum_{i\in[r]}\weight{V_i}=(2+\smallo{1})s.
\end{align*}
Let $I$ be the set of unordered pairs $i\neq j$ such that there is no edge in $H$ going between $V_i$ and $V_j$. We note that $K_\ell\nsubseteq H$, as $K_\ell$ is not a minor of $P$. Then by \Cref{lem:weighted_turan} we obtain
\begin{align*}
\sum_{\left\{i,j\right\}\in I}|V_i||V_j|\geq \frac{S_r}{2}\left(\frac{S_r}{\ell-1}-M_r\right)=\Th{s^2}.
\end{align*}
By definition of $I$ there is no edge in $P$ going between $V_i$ and $V_j$ for each unordered pair $\left\{i,j\right\}\in I$. Furthermore, \whp\ for each of these pairs at least $B=|V_i||V_j| s/\left(2n^2\right)$ many edges between $V_i$ and $V_j$ have been queried up to step $t$. Thus, the number of rejected edges satisfies that \whp 
\begin{align*}
\rej{t}\geq \frac{s}{2n^2} \sum_{\left\{i,j\right\}\in I}|V_i||V_j|\geq \Th{s^3/n^2}.
\end{align*}
This concludes the case $s\gg n^{17/24}$ and therefore also the case $t=n/2+s$ for $n^{2/3}\ll s\ll n$.

(ii) We consider the case where $n=cn/2$ for $c>1$. Let $\bar{t}:=n \cdot \log n$, $\bar{P}:=\pro{n}{\bar{t}}$, and $P:=\pro{n}{t}$. We claim that \whp\
\begin{align}\label{eq:5}
	\numberEdges{\bar{P}}\leq \left(1+\smallo{1}\right)n.
\end{align}
To show the claim, we use an idea from \cite{GerkeSchlatterStegerTaraz2008}: Let $\delta>0$, $m:=(1+\delta)n$, $H\in\cl(n,m)$, $\edgeSet{H}=\left\{f_1, \ldots, f_m\right\}$, and as in \Cref{subsec:motivation} $N:=\binom{n}{2}$. We have
\begin{align*}
\prob{H\subseteq\bar{P}}
~\leq~\prob{\left\{f_1, \ldots, f_m\right\}\subseteq\left\{e_1, \ldots, e_{\bar{t}}\right\}}
~\leq~\prod_{i=1}^{m}\prob{f_i\in \left\{e_1, \ldots, e_{\bar{t}}\right\}}
~=~\left(\frac{\bar{t}}{N}\right)^m
~=~\left(\frac{2\log n}{n-1}\right)^m.
\end{align*}
By \Cref{thm:small_classes} there exists a $c>0$ such that $\left|\cl(n)\right|\leq n!c^n$ for all $n\in\N$. Thus, we obtain by taking the union bound
\begin{align*}
\prob{\numberEdges{\bar{P}}\geq m}&~\leq~ \left|\cl(n)\right|\left(\frac{2\log n}{n-1}\right)^m
~\leq~
n!c^n\left(\frac{2\log n}{n-1}\right)^m
\\
&~\leq~
n^nc^n\left(\frac{2\log n}{n-1}\right)^m
~=~
\left(\frac{nc\left(2\log n\right)^{1+\delta}}{\left(n-1\right)^{1+\delta}}\right)^n
~=~\smallo{1},
\end{align*}
which gives \eqref{eq:5}. Next we use the well-known fact that \whp\ $\er{n}{\bar{t}}$ is connected (see e.g., \cite{ErdoesRenyi1959}). By \Cref{rem:component_structure} this is also true for $\bar{P}$. Together with \eqref{eq:5} this implies that \whp\ $\ex{\bar{P}}=\smallo{n}$. As $P\subseteq\bar{P}$, we obtain \whp\ $\ex{P}=\smallo{n}$. Due to $P\subseteq G$ we have $\numberEdges{G}-\numberEdges{P}\geq \ex{G}-\ex{P}$. Hence, we have that \whp\
\begin{align*}
t-\numberEdges{P}
~=~\numberEdges{G}-\numberEdges{P}
~\geq~\ex{G}-\ex{P}
~=~\left(c-\func{c}+\smallo{1}\right)n/2-\smallo{n},
\end{align*}
where we used \Cref{thm:erInt}\ref{thm:erIntB} for the last equality. This shows the lower bound in the case $t=cn/2$ for $c>1$ and concludes the proof of \Cref{thm:main}.
\qed

\section{Addable and forbidden edges and proof of \Cref{coro:main}}\label{sec:proof_coro}
How likely is it that the \lq next\rq\ edge $e_{t+1}$ gets accepted? Equivalently, what is the number of potential edges that can be added to $\pro{n}{t}$ without violating property $\cl$?

\begin{definition}
Let $\cl$ be a class of graphs and let $H\in \cl(n)$. Then we call an edge $e\in \edgeSet{K_n}\setminus \edgeSet{H}$ \textit{addable} to $H$ if $H+e\in \cl$ and \textit{forbidden} in $H$ otherwise, i.e., if $H+e\notin \cl$. Furthermore, we set
\begin{align*}
	\add{H}&:=\left|\setbuilder{e\in \edgeSet{K_n}\setminus \edgeSet{H}}{e \text{ is addable to }H}\right|;\\
	\forb{H}&:=\left|\setbuilder{e\in \edgeSet{K_n}\setminus \edgeSet{H}}{e \text{ is forbidden in } H}\right|.
\end{align*}
\end{definition}

In the next theorem we determine the number of forbidden edges in $\pro{n}{t}$. Combining it with \Cref{thm:main} one can also compute $\add{\pro{n}{t}}$, because
\begin{align*}
\add{\pro{n}{t}}=N-\forb{\pro{n}{t}}-\numberEdges{\pro{n}{t}}.
\end{align*}

\begin{thm}\label{thm:forbidden}
Let $\cl$ be a class of graphs satisfying the properties \ref{def:classA}--\ref{def:classF} in \Cref{def:class} and $\left(\pro{n}{t}\right)_{t=0}^{N}$ be the $\cl$-constrained random graph process. Let $h=h(n)=\smallomega{1}$ be a function which tends to $\infty$ arbitrarily slowly as $n\to\infty$. Let $t=t(n)\in\left[N\right]$ and $s=s(n)$. Then \whp
	\begin{align*}
	\forb{\pro{n}{t}}=
	\begin{cases}
		\bigo{hn^2/s} & \text{if} ~~ t=n/2-s ~~\text{for}~~ s\gg n^{2/3};
		\\
		\bigo{hn^{4/3}}& \text{if} ~~ t=n/2+s ~~\text{for}~~ s=\bigo{n^{2/3}};
		\\
		\Th{s^2}& \text{if} ~~ t=n/2+s ~~\text{for}~~ n^{2/3}\ll s\ll n;
		\\
		\left(\sol{c}^2+\smallo{1}\right)n^2/2& \text{if} ~~ t=cn/2 ~~\text{for}~~ c>1.
	\end{cases}
\end{align*}
\end{thm}

We note that for planar graphs (and many other graph classes mentioned in \Cref{prop:graph_classes}) we actually have that \whp\ 
\begin{align*}
\forb{\pro{n}{t}}=0 \quad \text{if }~ t=n/2-s ~\text{ for }~ s\gg n^{2/3}.
\end{align*}
However, this is not true in general for a class $\cl$ that satisfies the properties \ref{def:classA}--\ref{def:classF} in \Cref{def:class}. In order to prove \Cref{thm:forbidden}, we will use the following lemma. We recall that for fixed $n\in\N$ we denote by $\rej{t}=t-\numberEdges{\pro{n}{t}}$ the number of rejected edges up to step $t$.
\begin{lem}\label{lem:rejected}
Let $\cl$ and $\left(\pro{n}{t}\right)_{t=0}^{N}$ be as in \Cref{thm:forbidden}. Let $t_1=t_1(n), t_2=t_2(n)\in \left[N\right]$ and $\alpha=\alpha(n)>0$ be such that $\alpha \cdot \left(t_2-t_1\right)=\smallomega{1}$ and $t_2=\smallo{\alpha \cdot n^2}$. 
\begin{enumerate}
\item\label{lem:rejectedA}
If \whp\ $\rej{t_2}-\rej{t_1}\leq \alpha\cdot\left(t_2-t_1\right)$, then \whp\ $\forb{\pro{n}{t_1}}\leq \left(1+\smallo{1}\right)\alpha \cdot N$.
\item\label{lem:rejectedB}
If \whp\ $\rej{t_2}-\rej{t_1}\geq \alpha \cdot \left(t_2-t_1\right)$, then \whp\ $\forb{\pro{n}{t_2}}\geq \left(1+\smallo{1}\right)\alpha \cdot N$.
\end{enumerate}
\end{lem}
\begin{proof}
Due to property \ref{def:classD} of \Cref{def:class}, a forbidden edge in $\pro{n}{t}$ stays forbidden in $\pro{n}{t'}$ for all $t'>t$. Thus, $\forb{\pro{n}{t}}$ is non-decreasing in $t$. Let $\varepsilon>0$ be fixed and we denote by $\mathcal{F}$ the event that $\forb{\pro{n}{t_1}}\geq \left(1+\varepsilon\right)\alpha \cdot N$. Then we have that for each $t\in\left\{t_1+1, \ldots, t_2\right\}$ and $n$ large enough,
\begin{align*}
\condprob{e_t \text{ is rejected}}{\mathcal{F}}\geq \frac{\left(1+\varepsilon\right)\alpha \cdot N-t_2}{N}\geq \left(1+\varepsilon/2\right)\alpha,
\end{align*}
where we used $t_2=\smallo{\alpha \cdot n^2}$ in the last inequality. Hence, we obtain
\begin{align*}
\condprob{\rej{t_2}-\rej{t_1}\leq \alpha \cdot\left(t_2-t_1\right)}{\mathcal{F}}\leq \prob{\bin{t_2-t_1}{\left(1+\varepsilon/2\right)\alpha}\leq \alpha \cdot \left(t_2-t_1\right)}=\smallo{1}.
\end{align*}
Together with the fact that \whp\ $\rej{t_2}-\rej{t_1}\leq \alpha \cdot\left(t_2-t_1\right)$ it implies that $\prob{\mathcal{F}}=\smallo{1}$. As $\varepsilon>0$ was arbitrary, statement \ref{lem:rejectedA} follows. Statement \ref{lem:rejectedB} can be obtained similarly.
\end{proof}

Combining \Cref{thm:main} with \Cref{lem:rejected} we can prove \Cref{thm:forbidden}.
\proofofW{thm:forbidden}
(i) We start with the case $t=n/2+s$ for $n^{2/3}\ll s\ll n$. 

To prove the upper bound, we set $t_2=n/2+2s$. By \Cref{thm:main} we have that \whp\ $\rej{t_2}-\rej{t}=\bigo{s^3/n^2}$. Hence, there exists $\alpha=\alpha(n)$ such that $\alpha=\Th{s^2/n^2}$ and \whp\ $\rej{t_2}-\rej{t}\leq \alpha \cdot \left(t_2-t\right)$. Hence, \Cref{lem:rejected}\ref{lem:rejectedA} yields that \whp\ 
\begin{align*}
\forb{\pro{n}{t}}\leq \left(1+\smallo{1}\right)\alpha\cdot N=\Th{s^2}.
\end{align*}

Similarly, we obtain the lower bound: Taking $t_1=n/2$ we have by \Cref{thm:main} that \whp\ $\rej{t}=\Th{s^3/n^2}$ and $\rej{t_1}=\bigo{h}$ for $h=\smallomega{1}$, and thus $\frac{\rej{t}-\rej{t_1}}{t-t_1}=\Th{s^2/n^2}$. Together with \Cref{lem:rejected}\ref{lem:rejectedB} it implies that \whp\ $\forb{\pro{n}{t}}\geq \Th{s^2}$. 

(ii) The assertions in the cases $t=n/2-s$ for $s\gg n^{2/3}$ and $t=n/2+s$ for $s=\bigo{n^{2/3}}$ can be shown similarly.

(iii) Finally, we consider the regime $t=cn/2$ for $c>1$. Let $c_2>c$ and $t_2=c_2n/2$. By \Cref{thm:main} we have that \whp\
\begin{align*}
\rej{t_2}-\rej{t_1}=\left(c_2-\func{c_2}-c+\func{c}+\smallo{1}\right)n/2.
\end{align*}
Hence, \Cref{lem:rejected}\ref{lem:rejectedA} implies that \whp\ 
\begin{align*}
\forb{\pro{n}{t}}\leq \left(1+\smallo{1}\right)\left(1-\frac{\func{c_2}-\func{c}}{c_2-c}\right)n^2/2.
\end{align*}
With $c_2\downarrow c$ we obtain that \whp
\begin{align*}
\forb{\pro{n}{t}}\leq\left(1+\smallo{1}\right)\left(1-f'(c)\right)n^2/2=\left(1+\smallo{1}\right)\sol{c}^2n^2/2,
\end{align*}
where we used \Cref{lem:function}\ref{lem:functionA} in the last equality. Using \Cref{lem:rejected}\ref{lem:rejectedB} we obtain in a similar way that \whp\ $\forb{\pro{n}{t}}\geq\left(1+\smallo{1}\right)\sol{c}^2n^2/2$. This completes the proof.
\qed

We conclude this section by showing \Cref{coro:main}.
\proofofW{coro:main}
Let $C_1, \ldots, C_r$ be the components of $\pro{n}{t}$ such that $\numberVertices{C_1}\geq \ldots \geq \numberVertices{C_r}$. By \Cref{thm:erInt}\ref{thm:erIntA} and \ref{thm:erIntC} and \Cref{rem:component_structure} we have \whp\ $\numberVertices{C_1}=\left(\sol{c}+\smallo{1}\right)n$ and $\numberVertices{C_2}=\smallo{n}$. Furthermore, let $E_1\subseteq \edgeSet{K_n}\setminus\edgeSet{\pro{n}{t}}$ be the subset of edges with both endpoints in $\Largestcomponent$ and $E_2$ the remaining edges of $\edgeSet{K_n}\setminus\edgeSet{\pro{n}{t}}$. We have that \whp\
\begin{align}
\left|E_1\right|&=\left(\sol{c}^2+\smallo{1}\right)n^2/2 \label{eq:7};\\
\left|E_2\right|&=\left(1-\sol{c}^2+\smallo{1}\right)n^2/2. \label{eq:8}
\end{align}
Due to property \ref{def:classE} of \Cref{def:class} the two endpoints of a forbidden edge lie in the same component. Hence, the number of forbidden edges in $E_2$ is \whp\ at most
\begin{align*}
\numberVertices{C_2}^2+\ldots+\numberVertices{C_r}^2\leq \numberVertices{C_2}\left(\numberVertices{C_2}+\ldots+\numberVertices{C_r}\right)\leq \numberVertices{C_2}n=\smallo{n^2}.
\end{align*}
Together with \eqref{eq:8} it shows assertion \ref{coro:mainB}. Furthermore, it implies that the number of forbidden edges in $E_1$ is \whp\
\begin{align*}
\forb{\pro{n}{t}}-\smallo{n^2}=\left(\sol{c}^2+\smallo{1}\right)n^2/2.
\end{align*} 
Combining this with \eqref{eq:7} yields statement \ref{coro:mainA}. \qed

\section{The random graph $\acc{n}{m_0}$: proof of \Cref{thm:largestComponent}}\label{sec:accepted_graph}
Throughout this section, let $\cl$ be a class of graphs satisfying the properties \ref{def:classA}--\ref{def:classF} in \Cref{def:class} and $\left(\pro{n}{t}\right)_{t=0}^{N}$ be the $\cl$-constrained random graph process. Recall that $\acc{n}{m_0}$ denotes the graph in which exactly $m_0$ edges have actually been added. We assume that $m_0=m_0(n)\in\left[N\right]$ is such that $\acc{n}{m_0}$ always exists, i.e., for any ordering of the potential edges, at least $m_0$ of them have been accepted at the end of the process.

\proofofW{thm:largestComponent}
By \Cref{rem:component_structure} we have $\numberVertices{\largestcomponent{\pro{n}{t}}}=\numberVertices{\largestcomponent{\er{n}{t}}}$ and therefore we have that for $0\leq t_1\leq t_2$
\begin{align}\label{eq:6}
	\text{\whp}~~ \numberEdges{\pro{n}{t_1}}\leq m_0\leq \numberEdges{\pro{n}{t_2}}~~\implies~~ \text{\whp}~~ \numberVertices{\largestcomponent{\er{n}{t_1}}}\leq \numberVertices{\largestcomponent{\acc{n}{m_0}}}\leq \numberVertices{\largestcomponent{\er{n}{t_2}}}.
\end{align}

(i) First consider the case $m_0=cn/2$ for $1<c<2$. Let $\varepsilon>0$ be small, $t_1=\left(\invFunc{c}-\varepsilon\right)n/2$, and $t_2=\left(\invFunc{c}+\varepsilon\right)n/2$. By \Cref{thm:main} we have that \whp\ $\numberEdges{\pro{n}{t_1}}\leq m_0\leq \numberEdges{\pro{n}{t_2}}$. Thus, \eqref{eq:6} implies that \whp
\begin{align*}
\sol{\invFunc{c}-\varepsilon+\smallo{1}}n\leq \numberVertices{\largestcomponent{\acc{n}{m_0}}}\leq \sol{\invFunc{c}+\varepsilon+\smallo{1}}n,
\end{align*}
where we used \Cref{thm:erInt}\ref{thm:erIntA}. As $\beta$ is continuous, we obtain with $\varepsilon \downarrow 0$ that \whp\ $\numberVertices{\largestcomponent{\acc{n}{m_0}}}=\left(\sol{\invFunc{c}}+\smallo{1}\right)n$.

(ii) The four cases where $m_0\leq n/2+\smallo{n}$ can be shown similarly.

(iii) Next, we observe that $\numberVertices{\largestcomponent{\acc{n}{m_0}}}$ is non-decreasing in $m_0$, $\lim\limits_{c\uparrow 2}\invFunc{c}=\infty$, and $\lim\limits_{c\to\infty}\sol{c}=1$. Thus, in case of $m_0=n$ the statement follows by taking $c\uparrow 2$ in the previously shown fact that \whp\
\begin{align*}
\numberVertices{\largestcomponent{\acc{n}{m_0}}}=\left(\sol{\invFunc{c}}+\smallo{1}\right)n
\end{align*}
if $m_0=cn/2$ for $1<c<2$. 

(iv) Finally, we consider the case $m_0=cn/2$ for $2<c$ and let $t_1=n\cdot \log n$. By \eqref{eq:5} we know that \whp\ $\numberEdges{\pro{n}{t_1}}\leq m_0$. Thus, using \eqref{eq:6} yields 
\begin{align*}
\numberVertices{\largestcomponent{\acc{n}{m_0}}}\geq \numberVertices{\largestcomponent{\er{n}{t_1}}}=n,
\end{align*}
because \whp\ $\er{n}{t_1}$ is connected. This finishes the proof. \qed

\section*{Acknowledgement}

The authors thank the anonymous referees for many helpful remarks to
improve the presentation of this paper.

\bibliographystyle{plain}
\bibliography{kang-missethan-process}

\newpage
\appendix
\section{Properties of $f(c)$}\label{sec:propertiesF}
\begin{lem}\label{lem:function}
For $c>1$ let $\sol{c}$ be the unique positive solution of the equation $1-x=e^{-cx}$ and let $\func{c}=2\sol{c}+c\left(1-\sol{c}\right)^2$ (as in \Cref{def:func}). Then the following hold:
\begin{enumerate}
\item\label{lem:functionA}
$f'(c)=1-\sol{c}^2$;
\item\label{lem:functionB}
$f$ is strictly increasing;
\item\label{lem:functionC}
$f$ is continuous;
\item\label{lem:functionD}
$\lim\limits_{c\downarrow 1}f(c)=1$;
\item\label{lem:functionE}
$\lim\limits_{c\to \infty}f(c)=2$;
\item\label{lem:functionF}
$f:\left(1,\infty\right)\to \left(1,2\right)$ is bijective and therefore, has an inverse $f^{-1}:(1,2)\to(1,\infty)$.
\end{enumerate}
\end{lem}

\begin{proof}
By definition we have $1-\sol{c}=e^{-c\sol{c}}$. Taking the derivative of both sides yields
\begin{align*}
-\beta'(c)=\left(-\sol{c}-c\beta'(c)\right)e^{-c\sol{c}}=\left(-\sol{c}-c\beta'(c)\right)\left(1-\sol{c}\right).
\end{align*}
Rearranging this gives 
\begin{align*}
\beta'(c)=\frac{\sol{c}-\sol{c}^2}{1-c+c\sol{c}}.
\end{align*}
Using it we obtain
\begin{align*}
f'(c)&=2\beta'(c)+\left(1-\sol{c}\right)^2-2c\left(1-\sol{c}\right)\beta'(c)\\
&=
\frac{2\sol{c}-2\sol{c}^2+\left(1-\sol{c}\right)^2\left(1-c+c\sol{c}\right)-2c\left(1-\sol{c}\right)\left(\sol{c}-\sol{c}^2\right)}{1-c+c\sol{c}}
\\
&=1-\sol{c}^2,
\end{align*}
which shows \ref{lem:functionA}. As $\sol{c}\in \left(0,1\right)$ this implies also \ref{lem:functionB}. By definition $\beta$ is continuous and therefore so is $f$, proving \ref{lem:functionC}. We observe that $\lim\limits_{c\downarrow 1}\sol{c}=0$. Hence, we obtain that
\begin{align*}
\func{c}=2\sol{c}+c\left(1-\sol{c}\right)^2\to 2\cdot 0+1\cdot 1=1\quad \text{as} \quad c\downarrow 1,
\end{align*}
which shows \ref{lem:functionD}. 

For \ref{lem:functionE}, we note that $\lim\limits_{c\to \infty}\sol{c}=1$ and for $c$ large enough we have
\begin{align*}
ce^{-2\sol{c}c}\leq ce^{-c} \to 0 \quad \text{as} \quad c\to \infty.
\end{align*}
Thus, we obtain
\begin{align*}
\func{c}=2\sol{c}+ce^{-2\sol{c}c}\to 2\cdot 1+0=2 \quad \text{as}\quad c\to\infty,
\end{align*}
which shows \ref{lem:functionE}. 

Finally, \ref{lem:functionF} follows by combining \ref{lem:functionB}, \ref{lem:functionC}, \ref{lem:functionD}, and \ref{lem:functionE}.
\end{proof}
\begin{figure}
	\includegraphics[scale=0.4]{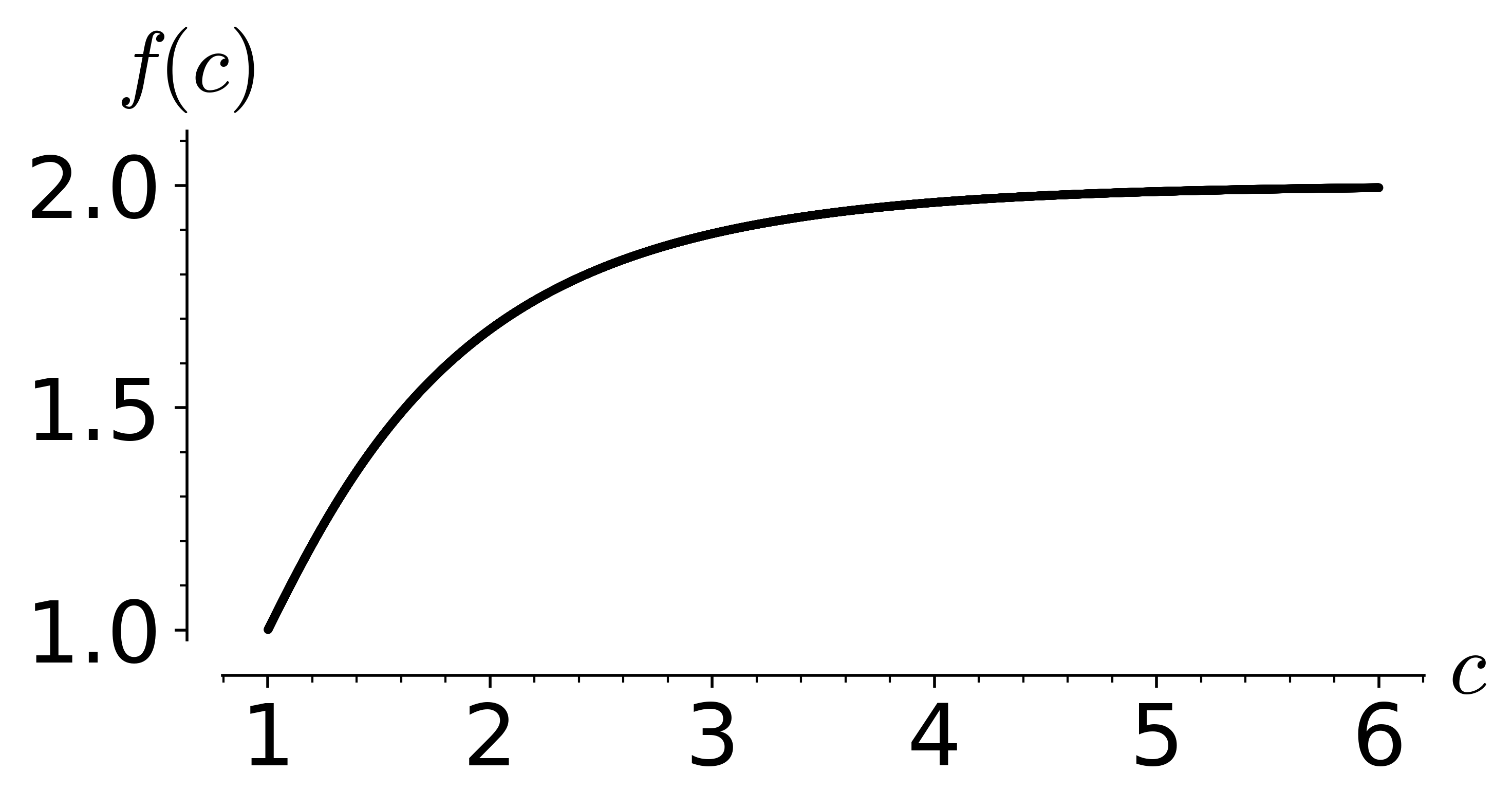}
	\caption{The function $f$ as defined in \Cref{def:func}.}
	\label{fig:function}
\end{figure}
\end{document}